\documentclass[a4paper,reqno]{amsart}

\usepackage{amsfonts}
\usepackage{mathrsfs}
\usepackage{graphicx}
\usepackage[normalem]{ulem}
\usepackage{url}
\usepackage{amsthm,amsfonts,amssymb,amsmath,esint,amscd,latexsym,amsxtra}
\usepackage[colorlinks=true, urlcolor=blue,bookmarks=true,bookmarksopen=true,
citecolor=blue]{hyperref}
\usepackage{amscd}
\usepackage{enumerate,bbm}

\usepackage{fullpage}
\usepackage[all]{xy}
\usepackage[OT2,T1]{fontenc}
\usepackage{xspace}

\newcommand{\BC}{{\mathbb {C}}} 
 
\newcommand{\BG}{{\mathbb {G}}}

 \newcommand{\BN}{{\mathbb {N}}}

 \newcommand{\BZ}{{\mathbb {Z}}}

\newcommand{\CC}{{\mathcal {C}}}

\newcommand{\CG}{{\mathcal {G}}} \newcommand{\CH}{{\mathcal {H}}}

\newcommand{\CM}{{\mathcal {M}}} 
\newcommand{\CO}{{\mathcal {O}}}

 \newcommand{\fp}{{\mathfrak{p}}}

\newcommand{\coker}{{\mathrm{coker}}}

 \newcommand{\EP}{{\mathrm{EP}}}

\newcommand{\Gal}{{\mathrm{Gal}}} \newcommand{\GL}{{\mathrm{GL}}}

\newcommand{\Hom}{{\mathrm{Hom}}}

\newcommand{\Rep}{{\mathrm{Rep}}}\newcommand{\RHom}{{\mathrm{RHom}}}

\newcommand{\Res}{{\mathrm{Res}}}

\newcommand{\Spec}{{\mathrm{Spec\hspace{2pt}}}}

\newcommand{\Ext}{{\mathrm{Ext}}}

\newcommand{\wh}[1]{{\widehat {#1}}}

\newcommand{\sk}{\medskip}

\newcommand{\bs}{\backslash}

\newcommand{\s}{\sk\noindent}

\newcommand{\dfn}[1]{\textit{#1}}

\makeatletter\def\varW@#1#2{%
\vtop{\m@th\ialign{##\cr
\hfil$#1 \mathrm{colim} $\hfil\cr
\noalign{\nointerlineskip\kern1.5\ex@}#2\cr
\noalign{\nointerlineskip\kern-\ex@}\cr}}
}
\makeatother
\makeatletter
\def\colim{%
\mathop{\mathpalette\varW@{}}\nmlimits@
}\makeatother

\theoremstyle{plain}
\newtheorem{thm}{Theorem}[section] \newtheorem{cor}[thm]{Corollary}

\newtheorem{lem}[thm]{Lemma}  \newtheorem{prop}[thm]{Proposition}
\newtheorem {conj}[thm]{Conjecture}

\newtheorem*{utf}{Unramified twisting family}
\theoremstyle{remark} \newtheorem{remark}[thm]{Remark}
\theoremstyle{definition} 
\theoremstyle{definition}  
\newtheorem{defn+lem}[thm]{Definition and Lemma}

\numberwithin{equation}{section}

\newcommand{\Mod}{\mathrm{Mod}}

\newcommand*{\sheafhom}{\mathrm{H}\kern -.5pt om}
\begin{document}
\title{Multiplicities of Representations in Algebraic Families}

\author{Li Cai}
\address{Academy for Multidisciplinary Studies\\
Beijing National Center for Applied Mathematics\\
Capital Normal University\\
Beijing, 100048, People's Republic of China}
\email{caili@cnu.edu.cn}

\author{Yangyu Fan}
\address{Academy for Multidisciplinary Studies\\
Beijing National Center for Applied Mathematics\\
Capital Normal University\\
Beijing, 100048, People's Republic of China}
\email{b452@cnu.edu.cn}


%\{1991 Mathematics Subject Classification. 22E45,20G25.}


\maketitle

\begin{abstract}
    In this short notes, we consider multiplicities of representations in general algebraic
    families, especially the upper semi-continuity  of homological multiplicities
    and the locally constancy of Euler-Poincare numbers. This generalizes the main result
    of Aizenbud-Sayag for unramified twisting families.  
\end{abstract}

\tableofcontents

\section{Introduction}
Let $G$ be a reductive group over a $p$-adic field $F$ and $H\subset G$ be a  closed \dfn{spherical} reductive subgroup, i.e. $H$ admits an open orbits  the flag variety of $G$. Let $\Rep(G,\BC)$ be the category of complex smooth $G(F)$-representations. In the \dfn{relative Langlands program} (see \cite{SV} etc), it is central to study  the \dfn{multiplicity} $m(\sigma):=\dim\Hom_{H(F)}(\sigma,\BC)$  for smooth admissible $\sigma\in\Rep(G,\BC)$.

As suggested in \cite{Pra18}, to study $m(\sigma)$,  it is more convenient to consider the homological multiplicities   $m^i(\sigma):=\dim\Ext^i_{H(F)}(\sigma,\BC)$ and the Euler-Poincare number $\EP(\sigma):=\sum_{i\geq0}(-1)^im^i(\sigma)$ simultaneously. Usually,  the Euler Poincare number $\EP(\sigma)$ is easier to control and in many  circumstances, one may expect to deduce results on $m(\sigma)$ from those of $\EP(\sigma)$.
For example, it is conjectured in \cite{Pra18} (see also  \cite[Conjectures 6.4,6.5]{Wan19}) that when the pair   $(G,H)$ is \dfn{strongly tempered}, i.e. the matrix coefficients of  tempered $G(F)$-representations are absolutely integrable on $H(F)$, then   $m(\sigma)=\EP(\sigma)$ for any irreducible tempered  $\sigma\in \Rep(G,\BC)$. Actually the stronger result $m^i(\sigma)=0$ for  $i>0$ is known  for the $(\GL_{n+1}\times \GL_n, \GL_n)$
-case (see \cite{CS21}), Bessel model for classical groups (see \cite{Che21}),  the triple product case (see \cite{CF21b}) and when   $H(F)$ is compact   (see \cite[Thm 2.14]{AAG12} etc) or  $\sigma$ is supercuspidal (see \cite[Remark 6.6]{Wan19}).

 % When supercuspidal $\sigma$, $m^i(\sigma)=0$ for $i>0$ (see \cite[Remark 6.6]{Wan19}). 
In this paper, we shall consider   variations of $m^i(\sigma)$ and $\EP(\sigma)$ in families.  Throughout  this paper, we assume the following working hypothesis:
\begin{center}
    the multiplicity $m(\sigma)$ is \dfn{finite} for all irreducible $\sigma\in\Rep(G,\BC)$.
\end{center}
This implies  that  $m^i(\sigma)$ and $\EP(\sigma)$ are all well-defined and finite for arbitrary finite length $\sigma\in \Rep(G,\BC)$ (see the discussion at the beginning of Section \ref{Homo mul}).
Note that this hypothesis is already known in many cases \cite[Theorem 5.1.5]{SV} and conjectured to hold for all spherical pairs.

To explain the flavor of the main result, let us start with the  \dfn{unramified twisting family}.
\begin{utf}Let $P\subset G$  be a parabolic subgroup with  Levi factor $M$ and take $\sigma\in \Rep(M,\BC)$ of  finite length. Attached to the data $(P,M,\sigma)$, one has the unramfied twisting family   $\left\{I_P^G(\sigma\chi)\Big|\chi\in \wh{M}\right\}$ where  $\wh{M}$ is  the complex torus parameterizing unramified characters of $M(F)$. Then as functions on the complex torus $\wh{M}$,
\begin{itemize}
    \item 
    $m(I_P^G(\sigma\chi))$ is \dfn{upper semi-continuous}, 
    i.e. for each $n\in\BN$, the set $\left\{\chi \in \wh{M} \Big| m(I_P^G\sigma\chi)\leq n\right\}$ is open
    (see \cite[Appendix D]{FLO12});
    \item  $\EP(I_P^G(\sigma\chi))$ is constant (see \cite[Theorem E]{AS20}).
\end{itemize}
\end{utf} 
% When attacking  arithmetic problems using $p$-adic methods, one  usually needs to consider 
%families of automorphic representations parameterized by the eigenvarieties, which are  defined over $p$-adic fields and   contain \dfn{non-smooth} classical points in general (see \cite{Bel08}).
For arithmetic applications such as $p$-adic special value formulae on eigenvarieties (see \cite{Dis19} etc), we are motivated   to consider  the following setting (following \cite{EH14,Dis20}): Fix a subfield $E\subset\BC$ and  let $R$ be a finitely generated reduced $E$-algebra. Let $\pi$ be a finitely generated smooth admissible torsion-free $R[G(F)]$-module, namely a finitely generated $R[G(F)]$-module  $\pi$ such that 
\begin{itemize}
	\item  any $v \in \pi$ is fixed by an open compact subgroup of $G(F)$;
	\item  the submodule
	$\pi^K\subset \pi$ of $K$-fixed elements is  finitely generated  over $R$ for any compact open subgroup $K \subset G(F)$;
	\item  $\pi$ is torsion-free as a $R$-module.
\end{itemize}
For any point $x\in \Spec(R)$, let $k(x)$ be the residue field and denote the category of smooth $G(F)$-representations over $k(x)$ by $\Rep(G,k(x))$.
For $\pi|_x:=\pi\otimes_{R}k(x)\in \Rep(G,k(x))$, set  $$m^i(\pi|_x):=\dim_{k(x)} \Ext^i_{H(F)}\left(\pi|_x,k(x)\right),\quad \EP(\pi|_x):=\sum_{i\geq0}(-1)^im^i(\pi|_x)$$
where the Ext-groups are computed in $\Rep(G,k(x))$. Note that by Proposition \ref{EPE} below, all the numbers $m^i(\pi|_x)$ and  $\EP(\pi|_x)$ are well-defined and finite under our running hypothesis. 

Supported by the results for unramified twisting families, we  propose the following conjecture:
\begin{conj}\label{Homo multi conj}With respect to the Zariski topology on $\Spec(R)$,  $m^i(\pi|_x)$	is upper semi-continuous  for each $i\in\BN$ and $\EP(\pi|_x)$ is locally constant.
	\end{conj}

\begin{remark}\label{Gal}
The following example  in \cite{CF21b} 
shows that the upper-semicontinuity is optimal to expect in general. 
Let $K/F$ be a quadratic field extension and $\theta\in\Gal(K/F)$ be the non-trivial element. The spherical pair $(G:=\BG_m\bs \Res_{K/F}\GL_2,H:= \BG_m\bs \GL_2)$ is not strongly tempered. Consider the  $G(F)$-representation 
$I_P^G(\chi)$ where $P$ is the parabolic subgroup consisting of upper triangular matrices and 
$\chi = (\chi_1,\chi_2)$ is a character of the Levi quotient $M(F) \cong (K^\times)^2/F^\times$. Then $m^i(I_P^G \chi)=0$, $i\geq2$ and 
\begin{itemize}
    \item $m(I_P^G \chi)\leq1$ with the equality holds iff $\chi_1|_{F^\times}=\chi_2|_{F^\times}=1$ or $\mu(\chi):=\chi_1\cdot (\chi_2\circ\theta)=1$;
    \item $m^1(I_P^G \chi)\leq1$ with the equality holds iff $\chi_1|_{F^\times}=\chi_2|_{F^\times}=1$ and $\mu(\chi)\neq 1$;
    \item $\EP(I_P^G \chi)\leq 1$  with the equality holds iff $\mu(\chi)=1$.
\end{itemize}
In particular, consider the family $I_P^G(\sigma \chi_\lambda)$ where
$\sigma = (\xi,1)$  with $\xi:\ F^\times\bs K^\times\to\BC^\times$ is a non-trivial character and 
$\chi_\lambda = (|\cdot|^\lambda, |\cdot|^{-\lambda})$, $\lambda \in \BC$. 
Then as functions of $\lambda$,  
$m^0(I_P^G(\sigma \chi_\lambda))$ and $m^1(I_P^G(\sigma \chi_\lambda))$ 
both jump at $\lambda=0$ while $\EP(I_P^G(\sigma \chi_\lambda))$ is constant.
\end{remark}
To state the main result, we need to introduce more notations. As the local analogue of classical points in eigen-varieties,  we fix   a Zariski dense subset $\Sigma\subset \Spec(R)$ of closed points. We say \dfn{the fiber rank of $\pi$  is locally constant on $\Sigma$} if for any open compact subgroup $K\subset G(F)$, the function $\dim_{k(x)}\pi^K|_x$ is locally constant on $\Sigma$. For any $x\in\Spec(R)$, denote by $(\pi|_x)^\vee$ the smooth dual of $\pi|_x$.
\begin{thm}\label{main}
Let $\pi$  be a finitely generated torsion-free smooth admissible $R[G(F)]$-module whose fiber rank is locally constant on $\Sigma$. Assume moreover there exists a finitely generated torsion-free admissible R[G(F)]-module $\tilde{\pi}$
such that
      for any $x\in\Sigma$, $\tilde{\pi}|_x\cong (\pi|_x)^\vee$.
Then Conjecture \ref{Homo multi conj} holds for $\pi$.
\end{thm}
Before explaining the proof, we make several remarks.
\begin{remark}For the unramified twisting family $\pi$,
\begin{itemize}
    \item the fiber rank is locally constant by construction;
    \item the underlying space $\wh{M}$ is connected and smooth;
    \item the family $\tilde{\pi}$ can be taken as
    $\left\{I_P^G(\sigma^\vee\chi^{-1}) \Big| \chi \in \wh{M} \right\}$.
\end{itemize}
Thus Theorem \ref{main}  covers unramified twisting families. 
\end{remark}
\begin{remark}The existence of the 'dual' module $\tilde{\pi}$ and the locally constancy of fiber rank are necessary for our approach (see the paragraphs after the remark). But these conditions do not impose very serious restrictions: 
\begin{itemize}
    \item the locally  constancy  of fiber rank may holds for general finitely generated torsion-free smooth admissible $R[G(F)]$-modules. If $\pi|_x$ is absolutely irreducible for all $x\in\Sigma$ and $G=\GL_n$, one can deduce the localy constancy of fiber rank from the theory of co-Whittaker modules in \cite{EH14};
\item if $\pi|_x$ is absolutely irreducible for all $x\in\Sigma$ and $G$ is classical, one can construct the $R[G(F)]$-module $\tilde{\pi}$  from $\pi$ by the MVW involution  (see  \cite{Pra19}).
\end{itemize}
\end{remark}

 Now we explain our approach to Theorem \ref{main}. We shall use the language of derived categories(See Section \ref{homo} for the basics). Let $i_H^G E$ be the compact induction of the trivial representation $E$ of $H(F)$ and $\CH(K,E)$ be the level-$K$ Hecke algebra  over $E$. 
 Then by the Frobenius reciprocity law and Bernstein's decomposition theorem (see \cite[Theorem 2.5(1)]{AS20} etc),  for properly chosen open compact subgroup $K\subset G(F)$ (see Proposition \ref{EPE} below)
 $$m^i(\pi|_x)=\dim_{k(x)}\Ext^i_{\CH(K,E)}((i_H^GE)^K,\tilde{\pi}^K|_x) = \dim_{k(x)} H^i\left(\RHom_{\CH(K,E)}( (i_H^GE)^K,
 \tilde{\pi}^K|_x)\right).$$
 By the following upper semi-continuous theorem, this simple observation reduces Theorem \ref{main} to
 \begin{itemize}
     \item the perfectness of $(i_H^GE)^K\in D(\CH(K,E))$, which we show in Proposition \ref{fdE} using the projective resolutions of $G(F)$-representations in \cite[Appendix]{Fli92};
     \item the projectiveness of the $R$-module $\tilde{\pi}^K$  by our assumption on the locally constancy of fiber rank (up to shrinking $R$, see Lemma \ref{proj} below).
 \end{itemize} 
 Here $D(\CH(K,E))$ is the derived category of $\CH(K,E)$-modules. 
 \begin{prop}[Upper semi-continuous theorem] \label{ST}
For any complex $M\in D(R)$, \begin{itemize}
    \item the function $\dim_{k(x)}H^i(M\otimes_R^Lk(x))$ is upper semicontinuous for each $i$ if $M$ is pseudo-coherent, i.e. quasi-isomorphic to a bounded above 
 complex of finite free $R$-modules;
    \item   the Euler Poincare number 
$\sum_i(-1)^i\dim_{k(x)}H^i(M\otimes_R^Lk(x))$ is locally constant if $M$ is perfect i.e. quasi-isomorphic to a bounded above 
and below complex of finite projective $R$-modules. 
\end{itemize}  
\end{prop}
\begin{proof}Item (i)  is
\cite[\href{https://stacks.math.columbia.edu/tag/0BDI}{Lemma 0BDI}]{SP} and Item (ii) is \cite[\href{https://stacks.math.columbia.edu/tag/0BDJ}{Lemma 0BDJ}]{SP}
\end{proof}
More precisely, Theorem \ref{main} holds if
\begin{itemize}
\item[(a)] $\RHom_{\CH(K,E)}((i_H^GE)^K, \tilde{\pi}^K)$ is perfect and 
\item[(b)] there is an isomorphism 
$$\RHom_{\CH(K,E)}((i_H^GE)^K, \tilde{\pi}^K)\otimes_R^L k(x)\cong \RHom_{\CH(K,E)}((i_H^GE)^K,\tilde{\pi}^K|_x).$$
\end{itemize}
As $\tilde{\pi}^K$ is projective,  the isomorphism in (b)
is equivalent to
$$\RHom_{\CH(K,E)}((i_H^GE)^K, \tilde{\pi}^K)\otimes_R^L k(x)\cong 
\RHom_{\CH(K,E)}((i_H^GE)^K,\tilde{\pi}^K\otimes_{R}^Lk(x)).$$
By the perfectness of  $(i_H^GE)^K\in D(\CH(K,E))$, the   isomorphism above holds
by standard  homological algebra (see Lemma \ref{base II}). 

With $(b)$ at hand,  by the general criterion of perfectness given in Lemma \ref{per II} below, the complex
$\RHom_{\CH(K,E)}((i_H^GE)^K, \tilde{\pi}^K)$ is perfect as $\Ext^i_{\CH(K,E)}((i_H^GE)^K, \tilde{\pi}^K)$
is finitely generated over $R$ (see Lemma \ref{Noeth}) 
and there is a positive integer $N$ such that for any closed point $x$, 
$m^i(\pi|_x) = 0$ for any $i \geq N$ (see Proposition \ref{EPE}).

\begin{remark}We briefly compare our approach with that in \cite{AS20}, which deals with  unramfied twisting families associated with $(P,M,\sigma)$.   By Frobenius reciprocity law,
$$\Ext^i_{H(F)}(I_P^G(\sigma\chi),\BC)\cong \Ext^i_{M}(r_M^G(i_H^G\BC), \sigma^\vee\chi^{-1})$$
where $r_M^G$ is the normalized Jacquet module functor from $\Rep(G,\BC)$ to $\Rep(M,\BC)$. In \textit{loc.cit}, the authors work over $\wh{M}$ and make full advantage of  the theory of Bernstein center and Bernstein decomposition to show
that there is a perfect complex $\CG(M,\sigma)$  over $\wh{M}$ associated to  $(M,\sigma^\vee)$ such that
\[\Ext^i_{M}(r_M^G(i_H^G\BC), \sigma^\vee\chi^{-1}) = H^i\left(\RHom_{\BC[\wh{M}]}( \CG(M,\sigma),\delta_\chi )\right)\]
where $\delta_\chi$ is  the skyscraper sheaf at $\chi^{-1}$. Then the locally constancy of Euler-Poincare numbers 
holds by the semicontinuity theorem for coherent sheaves over smooth varieties. 
In comparison, our approach seems more direct: 
\begin{itemize}
  \item it works over $G$ and does not depend on the special form of the family;
  \item  it requires less results from representation theory (while  more results from homological algebra).
\end{itemize}

\iffalse
\begin{itemize}
    \item the complex $i_{M_0}^{M}\rho^\vee\otimes^L_{\CH(M)}r_M^G(i_H^G\BC)$ on $\wh{M}$ is perfect, where $M_0=\cap_{\chi\in\wh{M}}\ker(\chi)$;
    \item for $x\in \wh{M}$ corresponding to $\chi^{-1}$,
    $$\RHom_{M}(r_M^G(i_H^G\BC), \sigma^\vee\chi^{-1})\cong \RHom_{\CO(\wh{M})}(i_{M_0}^{M}\rho^\vee\otimes^L_{\CH(M)}r_M^G(i_H^G\BC),k(x))$$
\end{itemize}
and then deduce the locally constancy of Euler-Poincare numbers by the semicontinuity theorem for coherent sheaves over smooth varieties (Proposition 5.4 in \textit{loc.cit}). 
\fi

\end{remark}
We conclude the introduction by the local constancy of $m(\pi|_x)$ for a finitely generated smooth admissible torsion-free $R[G(F)]$-module $\pi$.
\begin{cor}Assume 
the pair $(G,H)$ is strongly tempered and 
\begin{itemize}
\item the fiber rank of $\pi$  is locally constant on $\Sigma$ and there exists a finitely generated smooth admissible torsion-free $R[G(F)]$-module $\tilde{\pi}$ such that $\tilde{\pi}|_x\cong(\pi|_x)^\vee$ for any $x\in\Sigma$,
    \item for any $x\in \Sigma$, $(\pi|_x) \otimes_{k(x),\tau} \BC$ is irreducible and tempered for some field embedding
	$\tau: k(x) \hookrightarrow \BC$.
\end{itemize}
 Then $m(\pi|_x)$ is locally constant on $\Sigma$ if  $m(\sigma)=\EP(\sigma)$ holds for all irreducible tempered $\sigma\in \Rep(G,\BC)$.
 \end{cor}
We  remark that 
\begin{itemize}
    \item when $H(F)$ is compact,   the upper semi-continuity of multiplicities  holds under weaker  assumptions (see Proposition \ref{compact H-dist} below);
    \item   when  $(G,H)$ is strongly tempered and \dfn{Gelfand}, i.e. $m(\sigma)\leq 1$  for all $\sigma\in\Rep(G,\BC)$ irreducible,  the local constancy of multiplicities can be deduced from  the upper semi-continuity and the meromorphy property of canonical local periods considered in \cite{CF21a} (see also \cite{BD08} for analytic families). 
\end{itemize}

%We do not address the issue here but 
%\begin{itemize}
  %  \item  when pair In particular, the local constancy of multiplicity holds for the  GGP pair  $(U_{n+1}\times U_n, U_n)$.
  %  \item the local constancy of certain root numbers in various cases (including the GGP case). The local constancy of root numbers for families of %Weil-Deligne representations is   considered in \cite[Cor. 5.3.3]{Dis20}. 
%\end{itemize}
\section{Homological algebra}\label{homo}
 For any (unital but possibly noncommutative) ring $A$, denote by $\Mod_A$  (resp. $K(A)$) the category of left $A$-modules (resp. complexes of $A$-modules). The derived category $(D(A),q)$ consists a  category  $D(A)$  with a functor $q: K(A)\to D(A)$ such that any functor $F:\ K(A)\to\CC$ (to any category) which sends quasi-isomorphisms to isomorphisms factors uniquely through $q: K(A)\to D(A)$ (see \cite[Chapter III, Section 2]{GM03}). We usually   denote the derived category by $D(A)$.  
The  tensor product  and Hom functor  on $\Mod_A$ admit  derived version on $D(A)$ (see  \cite[Chapter 15]{SP} for $A$ commutative and \cite{Ye12} for general $A$). In particular, we record that
 \begin{itemize}
     \item for any $A$-algebra $A^\prime$,  viewed as a left $A^\prime$-module and right $A$-module, the tensor product   functor
$$A^\prime\otimes_A-:\ \Mod_A\to \Mod_{A^\prime},\quad M\mapsto A^\prime\otimes_AM$$
has the derived version $$A^\prime\otimes^L_A-:\ D(A)\to D(A^\prime).$$
Note that if $M\in D(A)$ is represented by a bounded above complex $P^\bullet\in K(A)$ of projective $A$-modules,  $A^\prime\otimes^L_A M$ is represented by $A^\prime\otimes_A P^\bullet$;
\item for any $N\in \Mod_A$, the  functor $$\Hom_A(-,N):\ \Mod_A\to \Mod_{\BZ};\quad M\mapsto \Hom_A(M,N)$$
has the derived version
$$\RHom_A(-,N):\ D(A)\to D(\BZ).$$
Note that if $M\in D(A)$ is represented by a bounded above complex $P^\bullet\in K(A)$ of projective $A$-modules, $\RHom_A(M,N)$ is represented by the complex
$\Hom(P^\bullet, N).$
Moreover, if $N\in \Mod_A$ admits a compatible left $R$-module structure for some commutative ring $R$,  the functor $\RHom_A(-,N)$ admits a natural lifting, which we denote by the same notation,   $$\RHom_A(-,N):\ D(A)\to D(R).$$
\end{itemize}
\iffalse Let $R$ be a commutative ring. Note that 
In particular, if $A$ is an algebra over $R$, then  $\RHom_A(-,-)$ admits a lifting, which we denote by the same notation,
$$\RHom_A(-,-):\ D(A)\times D(A)\to D(R).$$
\fi
 The following results on base change morphisms in derived category are crucial to our approach. Recall that a complex $M\in D(A)$ is called \dfn{pseudo-coherent} (resp.  \dfn{perfect}) if it is quasi-isomorphic to a bounded above (resp. above and below) complex of finite free (resp. projective) $A$-modules. Let $R$ be a commutative ring.
\begin{lem} \label{base I}Let $A^\prime$ be a flat $A$-algebra and take $N\in \Mod_{A^\prime}$. Then for any pseudo-coherent $M\in D(A)$, there is a canonical isomorphism $$\RHom_A(M,N)\cong \RHom_{A^\prime}(A^\prime\otimes_A^L M,N).$$
\end{lem}
\begin{proof}Assume $M$ is represented by the bounded above complex $P^\bullet\in K(A)$ of finite projective $A$-modules. Then 
$\RHom_A(M,N)$ is represented by $\Hom_A(P^\bullet, N)$
and $\RHom_{A^\prime}(A^\prime\otimes_A^L M,N)$ is represented by $\Hom_{A^\prime}(A^\prime\otimes_AP^\bullet, N)$. The desired result follows from  the  canonical isomorphism
$$\Hom_A(P, Q)\cong\Hom_{A^\prime}(A^\prime\otimes_AP, Q)$$
for any $A$-module $P$ and $A^\prime$-module $Q$ and the exactness of $A^\prime\otimes_A$. %Actually, much more general result holds(see \cite[\href{https://stacks.math.columbia.edu/tag/0E1W}{Lemma 0E1W}]{SP}).
\end{proof}
\begin{lem}\label{base II}Assume that $R$ is commutative and $N\in \Mod_A$ admits a compatible left $R$-module structure. Then for any perfect $M\in D(A)$  and $P\in D(R)$, one has the canonical isomorphism 
$$\RHom_A(M,N)\otimes_R^L P\to \RHom_A(M, N\otimes_R^LP).$$
\end{lem}
\begin{proof}%For case $(i)$, by \cite[\href{https://stacks.math.columbia.edu/tag/0ATI}{Lemma 0ATI}]{SP}, we are reduced to the case that $P$ is represented by the single term complex $R$, which is straightforward. For case $(ii)$, 
Represent $M$ by a bounded above and below complex $Q^\bullet$ of finite projective modules and $P$ by any complex $P^\bullet$. Then $\RHom_A(M,N)\otimes_R^LP$ is represented by the total complex $\mathrm{Tot}(\Hom_A(Q^\bullet,N)\otimes_R P^\bullet)$ and  $\RHom_A(M, N\otimes_R^LP)$ is represented by the  complex $\Hom_A^\bullet(Q^\bullet, N\otimes_R P^\bullet)$ with 
$$\Hom^n_A(Q^\bullet,N\otimes_R P^\bullet):=\prod_{n=p+q}\Hom_A(Q^{-p}, N\otimes_R P^q).$$ Note that for any $W\in \Mod_A$ finite projective and $V\in \Mod_R$, one has
$$\Hom_A(W,N)\otimes_R V\cong \Hom_A(W, N\otimes_R V).$$
Thus the complexes $\mathrm{Tot}(\Hom_A(Q^\bullet,N)\otimes_R P^\bullet)$ and $\Hom_A^\bullet(Q^\bullet, N\otimes_R P^\bullet)$ are isomorphic and we are done.
\end{proof}
\begin{lem}\label{base III}Assume $A$ is an algebra over $R$ and let $R^\prime$ be a flat commutative $R$-algebra. Then for  $M\in D(A)$  pseudo-coherent and $N\in\Mod_A$, the canonical morphism $$\RHom_A(M,N)\otimes_R^L R^\prime\to \RHom_{A\otimes_R R^\prime}(M\otimes_{R}^LR^\prime, N\otimes_R^L{R^\prime})$$
is an isomorphism. In particular if $A$ is Noetherian,  one has natural isomorphism 
$$\Hom_A(M,N)\otimes_R R^\prime\to \Hom_{A\otimes_R R^\prime}(M\otimes_{R}R^\prime, N\otimes_R{R^\prime})$$
for any  finitely generated  $M\in \Mod_A$.
\end{lem}
\begin{proof} Take a bounded above complex $P^\bullet$ of finite free $A$-modules representing $M\in D(A)$. Then 
$\RHom_A(M,N)\otimes_R^L R^\prime$ is represented by 
$\Hom_A(P^\bullet, N)\otimes_R R^\prime$
and $\RHom_{A\otimes_R R^\prime}(M\otimes_{R}^LR^\prime, N\otimes_R^L{R^\prime})$ is represented by $\Hom_{A\otimes_R R^\prime}(P^\bullet\otimes_{R}R^\prime, N\otimes_R{R^\prime})$.  The desired result follows from  the  canonical isomorphism
$$\Hom_A(P, Q)\otimes_R R^\prime\cong\Hom_{A\otimes_R R^\prime}(P\otimes_RR^\prime, Q\otimes_R R^\prime)$$
for any finite free $A$-module $P$ and arbitrary $A$-module $Q$.

By \cite[Lemma 064T]{SP}, any finitely generated $A$-module is pseudo-coherent when $A$ is Noetherian and the 'in particular' part  follows.
%Actually the results holds in much more general context (see \cite[\href{https://stacks.math.columbia.edu/tag/0A6A}{Lemma 0A6A}(3)]{SP}). 
    \end{proof}
Now we turn to perfect complexes over commutative rings. Let $R$ be a commutative Noetherian ring.
For   any $x\in\Spec(R)$, let $k(x)$ be the residue field.  \iffalse
\begin{lem}\label{per I}Assume $R$ is regular. Then $M\in \Mod_R$ is  perfect  in $D(R)$ iff $M$ is finitely generated.
    \end{lem}
    \begin{proof}See \cite[\href{https://stacks.math.columbia.edu/tag/066Z}{Lemma 066Z}]{SP}.
        \end{proof}
        \fi
       \begin{lem}\label{per II}A complex $M\in D(R)$  is perfect if the following conditions holds:
        \begin{enumerate}[(i)]
            \item the $R$-module $H^i(M)$ is finitely generated for each $i\in \BZ$;
            \item there exists $a<b\in\BZ$ such that for all closed point $x\in \Spec(R)$, $H^i(M\otimes_R^L k(x))=0$ if $i\notin [a,b]$.
        \end{enumerate}
         \end{lem}
    \begin{proof}   By \cite[\href{https://stacks.math.columbia.edu/tag/068W}{Lemma 068W}]{SP}, when $M$ is pseudo-coherent, Item (ii) implies $M$ is perfect.  By \cite[\href{https://stacks.math.columbia.edu/tag/064T}{Lemma 064T}]{SP} the assumption $R$ is Noetherian implies that all $H^i(M)$ is pseudo-coherent. 
   Thus  by \cite[\href{https://stacks.math.columbia.edu/tag/066B}{Lemma 066B}]{SP},  $M$ is pseudo-coherent if   $H^i(M)=0$ for all $i>b+1$. To see this, consider the exact sequence 
     $$0\to \fp^n/\fp^{n+1}\to R/\fp^{n+1}\to R/\fp^n\to0$$
     for any maximal ideal $\fp\subset R$. By induction, one finds that 
$H^i(M\otimes_{R}^LR/\fp^n)=0$ for all $i>b$.
By \cite[\href{https://stacks.math.columbia.edu/tag/0CQE}{Lemma 0CQE}]{SP} and  \cite[\href{https://stacks.math.columbia.edu/tag/0922}{Proposition 0922}]{SP}, one has the short exact sequence
$$0\to R^1\lim H^{i-1}(M\otimes_R^L R/\fp^n)\to H^i(R\lim M\otimes_R^L R/\fp^n)\to \lim H^i(M\otimes_R^L R/\fp^n)\to0.$$
By \cite[\href{https://stacks.math.columbia.edu/tag/0A06}{Lemma 0A06}]{SP}, one has 
$$H^i(M)\otimes_R\hat{R}_\fp=H^i(R\lim_n M\otimes_R^L R/\fp^n).$$
Thus $H^i(M)\otimes_R\hat{R}_\fp=0$ for all $i>b+1$ and all maximal ideal $\fp\subset R$. Consequently, $H^i(M)=0$ for all $i>b+1$ and we are done.
\end{proof}
Finally, we record the following result for a commutative Noetherian ring $R$.
\begin{lem}\label{proj} For any finitely generated $R$-module $M$,  the fiber rank function $$\beta(x):\ \Spec(R)\to \BN;\quad x\mapsto \dim_{k(x)}M\otimes_R k(x)$$ is upper-semicontinuous.
If  $R$ is moreover reduced, then  $M$ is projective iff $\beta$ is locally constant.
\end{lem}
\begin{proof}The first part follows from Proposition \ref{ST}. For the second part, see  \cite[\href{https://stacks.math.columbia.edu/tag/0FWG}{Lemma 0FWG}]{SP}.
    \end{proof}

\section{Homological  multiplicities}\label{Homo mul}
Let $(G,H)$ be a spherical pair of reductive groups over $p$-adic field $F$. Let $I_H^G\BC$ be  the space $$\{f:\ G(F)\to\BC\ \mathrm{smooth}\mid f(hg)=f(g)\ \forall\ h\in H(F),\ g\in G(F)\}$$ on which $G(F)$ acts by right translation and
$i_H^G\BC\subset I_H^G\BC$ be the subspace consisting of functions which are compactly supported modulo $H(F)$.  Since $H(F)$ is unimodular,  $I_H^G\BC$ and $i_H^G\BC$ are just the normalized induction and normalized compact induction  of the trivial representation $\BC$  of $H(F)$ respectively.

By  \cite[Proposition 2.5]{Pra18},  $$m^i(\sigma)=\dim_\BC\Ext^i_{G(F)}(\sigma,I_H^G\BC)=\dim_\BC\Ext^i_{G(F)}(i_H^G\BC,\sigma^\vee),\quad \forall\ \sigma\in\Rep(G,\BC).$$ 
 For any compact open subgroup $K\subset G(F)$, let $\CH(K,\BC)$ be the Hecke algebra of $\BC$-valued bi-$K$-invariant Schwartz functions on $G(F)$. Then by Bernstein's decomposition theorem (see \cite[Theorem 2.5(1)]{AS20} etc), there exists a  neighborhood basis $\{K\}$ of $1\in G(F)$ consisting of  \dfn{splitting} (see \cite{AAG12} for the notation) open compact subgroups  such that 
   $\CH(K,\BC)$ is Noetherian, the subcategory $\CM(G,K,\BC)$ of representations generated by their $K$-fixed vectors is a direct summand of $\Rep(G,\BC)$ and the functor $\sigma\mapsto \sigma^K$ induces an equivalence of categories $\CM(G,K,\BC)\cong \Mod_{\CH(K,\BC)}$. Thus for
 $\sigma^\vee\in M(G,K,\BC)$ with $K$ splitting,
$$\Ext_{\CH(K,\BC)}^i((i_H^G\BC)^K,(\sigma^\vee)^K)\cong \Ext_{G(F)}^i(i_H^G\BC,\sigma^\vee),\quad \forall\ i\in\BZ.$$
 Under our working hypothesis \begin{center}
    the multiplicity $m(\sigma)$ is \dfn{finite} for all irreducible $\sigma\in\Rep(G,\BC)$, 
\end{center}
 $i_H^G\BC$  is {\em locally finitely generated}, i.e. for any compact open subgroup $K\subset G(F)$,  $(i_H^G\BC)^K$  is finitely generated over $\CH(K,\BC)$, by   \cite[Theorem A]{AAG12}. Thus for $K$ splitting, $(i_H^G\BC)^K$ admits a resolution by finite projective $\CH(K,\BC)$-modules and consequently for  $\sigma\in\Rep(G,\BC)$ such that $\sigma^\vee\in M(G,K,\BC)$,  $$m^i(\sigma)=\dim_\BC\Ext^i_{\CH(K,\BC)}((i_H^G\BC)^K,(\sigma^\vee)^K)<\infty,\quad \forall i\in\BN.$$
 In particular, for any $\sigma\in\Rep(G,\BC)$ of finite length,   $m^i(\sigma)$ is finite for all $i$. Finally by \cite[Corollary III.3.3]{Pra18}, for any finite length $\sigma\in\Rep(G,\BC)$, $m^i(\sigma)=0$ for $i>d(G)$, the split rank of $G$. 
 %In summary,
%\begin{prop}\label{EP number}For any $\sigma\in\Rep(G,\BC)$ of finite length,  the homological multiplicity  $m^i(\sigma)$ is finite for each $i\geq0$ and  $0$ for $i>d(G)$. In particular,
%	$\EP(\sigma)$ is actually a finite sum.  
%\end{prop}

Now we change the coefficient field $\BC$ to a  subfield $E\subset\BC$.   Let $\Rep(G,E)$ be category of smooth $G(F)$-representations over $E$. For any open compact subgroup $K\subset G(F)$,  let $\CH(K,E)$ be the Hecke algebra of $E$-valued bi-$K$-invariant Schwartz functions on $G(F)$,
    and $M(G,K,E)\subset \Rep(G,E)$ be the subcategory of representations generated by their $K$-fixed vectors.
 For  any $\sigma\in \Rep(G,E)$, set $$m^i(\sigma):=\dim_E \Ext^i_{H(F)}(\sigma,E),\ \forall\ i\in\BN,\quad \EP(\sigma):=\sum_i(-1)^im^i(\sigma).$$ 
  Let $i_H^GE\in\Rep(G,E)$ be the compact induction of the trivial $H(F)$-representation $E$. For any  open compact subgroup $K\subset G(F)$, 
  \begin{lem}\label{Noeth}For any splitting open compact subgroup $K\subset G(F)$, the Hecke algebra $\CH(K,E)$ is Noetherian and $(i_H^GE)^K$ is finitely generated over  $\CH(K,E)$.
  \end{lem}
  \begin{proof}
Note that $(i_H^GE)^K\otimes_{E}\BC=(i_H^G\BC)^K$. Take generators $\{ y_i=\sum_j f_{i,j}\otimes a_{i,j}\}$ of $(i_H^G\BC)^K$ over $\CH(K,\BC)$ with $f_{i,j}\in (i_H^GE)^K$.  Let $V\subset i_H^GE$ be the $\CH(K,E)$-submodule generated by $f_{i,j}$. Since $y_i$ belongs to $V\otimes_E \BC$ for each $i$, one has 
$V\otimes_E \BC= (i_H^GE)^K\otimes_{E}\BC$. Consequently, $V=i_H^GE$ and $i_H^GE$ is locally finitely generated.

Take any ascending chain of left ideals of $\CH(K,E)$
	$$I_0\subset I_1\subset\cdots\subset I_n\subset\cdots.$$
	Then $I_i\otimes_{E}\BC$ forms an ascending chain of left ideals of $\CH(K,\BC)\cong \CH(K,E)\otimes_{E}\BC$. Since $\CH(K,\BC)$ is Noetherian, we have that for some $n$, $$I_n\otimes_{E}\BC=I_{n+1}\otimes_{E}\BC=\cdots.$$
 Consequently, $I_n=I_{n+1}=\cdots$	 and $\CH(K,E)$ is Noetherian.
   \end{proof}

  \begin{prop}\label{EPE}  For any $\sigma\in \Rep(G,E)$  such that $\sigma^\vee\in M(G,K,E)$,  the homological multiplicity  $$m^i(\sigma)=\dim_E \Ext^i_{\CH(K,E)}( (i_H^GE)^K,(\sigma^\vee)^K)\quad \forall\ i\in\BN.$$ If moreover $\sigma$ has finite length, then $m^i(\sigma)$ is finite for each $i\geq0$ and $0$ for $i>d(G)$. In particular,  $\EP(\sigma)$
	is actually a  finite sum.   
	\end{prop}
\begin{proof} For any  $\sigma\in \Rep(G,E)$, set $\sigma_\BC:=\sigma\otimes_{E}\BC$.
  Then for any $\sigma\in\Rep(G,E)$ and  $\theta\in\Rep(G,\BC)$  $$\Hom_{G(F)}(\sigma,\theta)=\Hom_{G(F)}(\sigma_\BC,\theta).$$
 Thus computing using any projective resolution of $\sigma$, one finds 
  $$\Ext^i_{G(F)}(\sigma,I_H^GE)\otimes_E\BC\cong\Ext^i_{G(F)}(\sigma_\BC,I_H^G\BC)\quad \forall\ i\geq0.$$
By Lemma \ref{Noeth},  $(i_H^GE)^K\in D(\CH(K,E))$ is pseudo-coherent for $K$ splitting.  Thus by Lemma \ref{base III},
  $$\Ext^i_{\CH(K,E)}((i_H^GE)^K,(\sigma^\vee)^K)\otimes_{E}\BC\cong \Ext^i_{\CH(K,\BC)}((i_H^G\BC)^K,(\sigma_\BC^\vee)^K)\quad \forall\ i\geq0.$$
 From the corresponding results for $\sigma_\BC$, one deduce that
 \begin{itemize}
     \item if $\sigma^\vee\in\Rep(G,K,E)$, $$m^i(\sigma)=\dim_E \Ext^i_{\CH(K,E)}((i_H^GE)^K,(\sigma^\vee)^K),\ \forall\ i\geq0,$$
     \item if $\sigma$ has finite length,  $m^i(\sigma)$  is finite for all $i\geq0$ and $m^i(\sigma)=0$ if $i>d(G)$.
 \end{itemize}
\end{proof}

\begin{prop}\label{fdE}For any splitting open compact subgroup $K\subset G(F)$, $(i_H^GE)^K\in D(\CH(K,E))$ is perfect.
\end{prop}
\begin{proof} Take  $V\subset i_H^G\BC$ be the sub-representation generated by $(i_H^G\BC)^K$. By \cite[Appendix]{Fli92}, $V$ admits an explicit bounded above and below  resolution by projective objects in $\Rep(G,\BC)$ (actually for certain $K$, the projective resolution can be made explicitly by the theory of Schneider-Stuhler, see \cite[Theorem II.3.1]{SS98} and \cite[Theorem 1.2]{OS13}).
Thus there exists a positive integer $N$ such that for any $W\in \CM(G,K,\BC)$, $$\Ext^i_{G}(V,W)=0,\quad \forall\ i>N.$$  
Note that $\sigma\mapsto \sigma^K$ induces an equivalence between $\CM(G,K,\BC)$ and the category of $\CH(K,\BC)$-modules (see \cite[Theorem 2.5(1)]{AS20}). Since $(i_H^G\BC)^K=V^K$, one finds  $$\Ext^i_{\CH(K,\BC)}((i_H^G\BC)^K,M)=0,\quad \forall\ i>N$$
for  any $\CH(K,\BC)$-module $M$.
Thus by Lemma \ref{base III}, 
for  any $\CH(K,E)$-module $M$, $$\Ext^i_{\CH(K,E)}((i_H^GE)^K,M)=0,\quad \forall i>N.$$
Take any resolution $P^\bullet$ of $(i_H^GE)^K$ $$\cdots\to P_{N+1}\to  P_N\to \cdots P_1\to P_0\to0\cdots$$  by finite projective $\CH(K,E)$-modules. 
Let $Q=\coker(P_{N+2}\to P_{N+1})$.
Then $Q$ admits a resolution $$\cdots\to P_{N+2}\to P_{N+1}\to 0\cdots$$
Then $\Ext^1_{\CH(K,E)}(Q,M) =0$ for $M=\ker(P_{N+1}\to Q)$. Consequently, $P_{N+1}=Q\oplus M$ and $Q$ is projective.
 Consequently, $(i_H^GE)^K$ is perfect as it is quasi-isomorphic to
$$\cdots0\to Q\to  P_N\to \cdots \to P_0\to 0\cdots.$$
\end{proof}

Now we  prove Theorem \ref{main}.  Let
$R$ be a finitely generated reduced $E$-algebra and $\Sigma\subset\Spec(R)$ be a Zariski denset subset of closed points. We restate Theorem \ref{main} for the convenience of  readers.

%We now formulate the conjectural behavior of $m^i(\pi)$ and $\EP(\pi)$ in families. Recall that $X$ is a locally of finite type $E$-scheme and $\Sigma\subset X$ is a Zariski dense subset of closed points. 
%\begin{conj}\label{Homo multi conj}Assume $i_H^GE$ is locally finitely generated. Let $\pi$ be a finitely generated smooth admissible torsion-free $\CO_X[G(F)]$-module such that $\pi|_x$ is absolutely irreducible for any $x\in\Sigma$. Then for each $i\in\BN$, the function 
%	$$X\to\BZ;\quad	x\mapsto m^i(\pi|_x)$$
%	is upper semi-continuous on $\Sigma$. Moreover, the function 
%	$$X\to\BZ;\quad x\mapsto \EP(\pi|_x)$$
%	is locally constant on $\Sigma$.
%	\end{conj}
%We have some partial results. Let $Z\subset G$ be the connected center.
%\begin{prop}\label{Upper semicontinuous multi}Assume $i_H^GE$ is locally finitely generated. Let $\pi$ be a finitely generated smooth admissible torsion-free $\CO_X[G(F)]$-module such that $\pi|_x$ is  irreducible for any $x\in\Sigma$  and there exists a finitely generated smooth admissible torsion-free $\CO_X[G(F)]$-module $\tilde{\pi}$ such that 
%	\begin{itemize}
%		\item for any $x\in\Sigma$, $\tilde{\pi}|_x\cong(\pi|_x)^\vee$,
%		\item the fiber rank of $\pi^K$ is locally constant on $\Sigma$ for any $K\subset G(F)$.
%	\end{itemize}   Then  Conjecture \ref{Homo multi conj} holds for $\pi$ if $X$ is regular or $Z(F)$ is compact.
%\end{prop}

\begin{thm}Let $\pi$ be a torsion-free smooth admissible finitely generated $R[G(F)]$-module whose fiber rank is locally constant on $\Sigma$. Assume that  there exists a finitely generated smooth admissible torsion-free $R[G(F)]$-module $\tilde{\pi}$ such that $\tilde{\pi}|_x\cong(\pi|_x)^\vee$ for any $x\in\Sigma$. Then $m^i(\pi|_x)$	is upper semi-continuous  for each $i\in\BN$ and $\EP(\pi|_x)$ is locally constant.
\end{thm}
\begin{proof}Since splitting open subgroups form an neighborhood system of $1\in G(F)$,  one can take a splitting  open compact subgroup $K\subset G(F)$ such that $\tilde{\pi}^K$ generates $\tilde{\pi}$ and $(i_H^GE)^K\in D(\CH(K,E))$ is perfect by Proposition  \ref{fdE}. Thus  by  Proposition \ref{base II},
$$\RHom_{\CH(K,E)}((i_H^GE)^K,\tilde{\pi}^K)\otimes_R^Lk(x)\cong \RHom_{\CH(K,E)}((i_H^GE)^K,\tilde{\pi}^K\otimes_R^Lk(x))$$
By the duality between $\pi^K|_x$ and $\tilde{\pi}^K|_x$ and Lemma \ref{proj},  upon shrinking $\Spec(R)$ to an open subset containing $\Sigma$ we can and will assume  the fiber rank of $\tilde{\pi}^U$ is locally constant on $\Sigma$ and thus  the $R$-module $\tilde{\pi}^U$ is finite projective. 
Thus 
$$\RHom_{\CH(K,E)}((i_H^GE)^K,\tilde{\pi}^K\otimes_R^Lk(x))\cong  \RHom_{\CH(K,E)}((i_H^GE)^K,\tilde{\pi}^K|_x).$$
 By Lemma \ref{base I} and Proposition \ref{EPE}, one has $m^i(\pi|_x)=\dim_{k(x)}\Ext^i_{\CH(K,E)}((i_H^G E)^K, \tilde{\pi}^K|_x).$ By Proposition \ref{ST}, to finish the proof it suffices to show the complex 
 $\RHom_{\CH(K,E)}((i_H^GE)^K,\tilde{\pi}^K)$ is perfect in $D(R)$. As  $(i_H^GE)^K$  admits a bounded above and below resolution  $P^\bullet$    by finite projective $\CH(K,E)$-modules, the complex $\RHom_{\CH(K,E)}((i_H^GE)^K, \tilde{\pi}^K)$ is represented by
$\Hom_{\CH(K,E)}(P^\bullet, \tilde{\pi}^K)$. Since $\tilde{\pi}^K$ is finitely generated over $R$ by the admissibility of $\tilde{\pi}$, $\Hom_{\CH(K,E)}(P^\bullet, \tilde{\pi}^K)$ is a complex of finitely generated $R$-modules. 
Thus $H^i(\RHom_{\CH(K,E)}((i_H^GE)^K, \tilde{\pi}^K))$ are finitely generated as $R$-modules for each $i\in\BZ$.
 Now the desired perfectness follows from Lemma \ref{per II}  and Proposition \ref{EPE}.
	\end{proof}
Finally, we remark that when $H(F)$ is compact, the upper semi-continuity of $m^i(\pi|_x)$ holds for all torsion-free finitely generated smooth admissible $R[G(F)]$-modules $\pi$ (here we do not assume the existence of $\tilde{\pi}$).
\begin{prop}\label{compact H-dist}Assume  $H(F)$ is compact.  Then for any torsion-free finitely generated smooth admissible $R[G(F)]$-module $\pi$,   the function $\EP(\pi|_x)= m^0(\pi|_x)$
is upper semi-continuous on $\Spec(R)$.
\end{prop}
\begin{proof} By \cite[Theorem 2.14]{AAG12}, $m^i(\pi|_x)=0$ for each $i\geq1$ and  $\EP(\pi|_x)=m^0(\pi|_x)$ for any $x\in \Spec(R)$.   Let $\pi_H$ (resp. $\pi^H$) be the $H(F)$-coinvariant (resp. $H(F)$-invariant) of $\pi$. Since $H(F)$ is compact, the natural map  $\pi^H\to \pi_H$ is an isomorphism and for any $x\in \Spec(R)$, $(\pi^H)|_x\cong (\pi|_x)^H\cong (\pi|_x)_H.$
In particular, $m^0(\pi|_x)=\dim_{k(x)}(\pi|_x)_H=\dim_{k(x)}\pi^H|_x$. By Lemma \ref{proj} to finish the proof,  it suffices to show $\pi^H$ is coherent.
 
 Note that  \cite[Theorem 2.5(1)]{AS20} actually works for any algebraically closed field of characteristic zero. Thus by Proposition \ref{EPE}, for each generic point $\eta$ of $\Spec(R)$ and some splitting open subgroup $K$
 $$\dim_{k(\eta)}\pi^H|_{\eta}=\dim_{k(\eta)}\Hom_{G(F)}((i_H^G k(\eta))^K, (\pi|_{\eta})^{\vee,K})<\infty.$$
Thus there exists an open compact subgroup $K^\prime\subset G(F)$ such that $\pi^H|_{\eta}\subset (\pi|_\eta)^{K^\prime}$ for all $\eta$.  Then for any $v\in \pi^H$ and $k\in K^\prime$, $k\cdot v= v$ in $\prod_{\eta}\pi|_\eta$. Since the diagonal map 
$ \pi\hookrightarrow \prod_{\eta}\pi|_\eta$ is injective, one has $k\cdot v=v$ in $\pi$ and consequently $\pi^H\subset \pi^{K^\prime}$. As $\pi^{K^\prime}$ is coherent by the admissibility of $\pi$,  $\pi^H$ is coherent and we are done.
 \end{proof}

\s{\bf Acknowledgement} The debt this work owes to \cite{Pra18} and \cite{AS20} is clear. We would like to thank Professor Ye Tian for his consistent encouragement. We also want to thank  the referee for the sharp reading, particularly directing us to the  reference \cite{Fli92} on the projective resolutions of smooth representations. 

This research is supported by the National Key R$\&$D Program of China No. 2023YFA1009702. L. Cai is also supported by National Natural Science Foudation of China No. 12371012. Y. Fan is also supported by  National Natural Science Foudation of Beijing, China No. 24A10020.

\end{document}